\documentclass{amsart}
\usepackage{amsmath}
\usepackage{amsfonts}

\newtheorem{theorem}{Theorem}
\theoremstyle{plain}

\newtheorem{corollary}{Corollary}

\newtheorem{definition}{Definition}

\newtheorem{lemma}{Lemma}

\newtheorem{proposition}{Proposition}

\numberwithin{equation}{section}
\input{tcilatex}

\begin{document}
\title[On $s$-Geometrically Convex Functions]{Some New Hermite-Hadamard type
inequalities for $s$-geometrically convex functions and their applications}
\author{\.{I}mdat \.{I}\c{s}can}
\address{Department of Mathematics, Faculty of Arts and Sciences,\\
Giresun University, 28100, Giresun, Turkey.}
\email{imdati@yahoo.com, imdat.iscan@giresun.edu.tr}

\begin{abstract}
In this paper, some new integral inequalities of Hermite-Hadamard type
related to the $s$-geometrically convex functions are established and some
applications to special means of positive real numbers are also given.
\end{abstract}

\maketitle

\section{Introduction}

In this section, we firstly list several definitions and some known results.

\begin{definition}
Let $I$ be an interval in $%
%TCIMACRO{\U{211d} }%
%BeginExpansion
\mathbb{R}
%EndExpansion
$. Then $f:I\rightarrow 
%TCIMACRO{\U{211d} }%
%BeginExpansion
\mathbb{R}
%EndExpansion
$ is said to be convex if%
\begin{equation*}
f\left( tx+(1-t)y\right) \leq tf(x)+(1-t)f(y)
\end{equation*}%
for all $x,y\in I$ and $t\in \lbrack 0,1]$.
\end{definition}

Let $f:I\subseteq \mathbb{R\rightarrow R}$ be a convex function defined on
the interval $I$ of real numbers and $a,b\in I$ with $a<b$. The following
double inequality is well known in the literature as Hermite-Hadamard
integral inequality \cite{DP00} 
\begin{equation}
f\left( \frac{a+b}{2}\right) \leq \frac{1}{b-a}\dint\limits_{a}^{b}f(x)dx%
\leq \frac{f(a)+f(b)}{2}\text{.}  \label{1-1}
\end{equation}%
In \cite{HM94}, Hudzik and Maligranda considered the following class of
functions:

\begin{definition}
A function $f:I\subseteq 
%TCIMACRO{\U{211d} }%
%BeginExpansion
\mathbb{R}
%EndExpansion
_{+}\rightarrow 
%TCIMACRO{\U{211d} }%
%BeginExpansion
\mathbb{R}
%EndExpansion
$ where $%
%TCIMACRO{\U{211d} }%
%BeginExpansion
\mathbb{R}
%EndExpansion
_{+}=\left[ 0,\infty \right) $, is said to be $s$-convex in the second sense
if%
\begin{equation*}
f\left( \alpha x+\beta y\right) \leq \alpha ^{s}f(x)+\beta ^{s}f(y)
\end{equation*}%
for all $x,y\in I$ and $\alpha ,\beta \geq 0$ with $\alpha +\beta =1$ and $s$
fixed in $\left( 0,1\right] $. They denoted this by $K_{s}^{2}.$
\end{definition}

It can be easily seen that for $s=1$, $s$-convexity reduces to ordinary
convexity of functions defined on $[0,\infty )$. For some recent results and
generalizations concerning $s$-convex functions see \cite%
{ADK11,AKO11,DF99,HBI09,I13b,KBO07}.

Recently, In \cite{ZJQ12}, the concept of geometrically and $s$%
-geometrically convex functions was introduced as follows:

\begin{definition}
A function $f:I\subset 
%TCIMACRO{\U{211d} }%
%BeginExpansion
\mathbb{R}
%EndExpansion
_{+}=\left( 0,\infty \right) \rightarrow 
%TCIMACRO{\U{211d} }%
%BeginExpansion
\mathbb{R}
%EndExpansion
_{+}$ is said to be a geometrically convex function if%
\begin{equation*}
f\left( x^{t}y^{1-t}\right) \leq f(x)^{t}f(y)^{1-t}
\end{equation*}%
for $x,y\in I$ and $t\in \lbrack 0,1]$.
\end{definition}

\begin{definition}
A function $f:I\subset 
%TCIMACRO{\U{211d} }%
%BeginExpansion
\mathbb{R}
%EndExpansion
_{+}=\left( 0,\infty \right) \rightarrow 
%TCIMACRO{\U{211d} }%
%BeginExpansion
\mathbb{R}
%EndExpansion
_{+}$ is said to be a $s$-geometrically convex function if%
\begin{equation*}
f\left( x^{t}y^{1-t}\right) \leq f(x)^{t^{s}}f(y)^{\left( 1-t\right) ^{s}}
\end{equation*}%
for some $s\in \left( 0,1\right] $, where $x,y\in I$ and $t\in \lbrack 0,1]$.
\end{definition}

In \cite{ZJQ12}, the authors has established some integral inequalities
connected with the inequalities (\ref{1-1}) for the $s$-geometrically convex
and monotonically decreasing functions. In \cite{XBQ12}, The authors have
introduced concepts of the $m$-and $(\alpha ,m)$-geometrically convex
functions and established some inequalities of Hermite--Hadamard type for
these classes of functions.

In \cite{I13}, the author proved the following results on the geometrically
convex functions.

\begin{theorem}
Suppose that $f:I\subseteq \mathbb{R}_{+}\mathbb{\rightarrow R}_{+}$ is
geometrically convex and $a,b\in I$ with $a<b$. If $f\in L\left[ a,b\right] $%
, then one has the inequalities:%
\begin{eqnarray*}
f\left( \sqrt{ab}\right) &\leq &\frac{1}{\ln b-\ln a}\dint\limits_{a}^{b}%
\frac{1}{x}\sqrt{f(x)f\left( \frac{ab}{x}\right) }dx\leq \frac{1}{\ln b-\ln a%
}\dint\limits_{a}^{b}\frac{f(x)}{x}dx \\
&\leq &\frac{f(b)-f(a)}{\ln f(b)-\ln f(a)}\leq \frac{f(a)+f(b)}{2}.
\end{eqnarray*}
\end{theorem}

\begin{lemma}
\label{1.1} Let $f:I\subseteq \mathbb{R}_{+}\mathbb{\rightarrow R}$ be a
differentiable mapping on $I^{\circ }$, and $a,b\in I$,with $a<b$. If $%
f^{\prime }\in L\left[ a,b\right] $, then%
\begin{equation*}
f\left( \sqrt{ab}\right) -\frac{1}{\ln b-\ln a}\dint\limits_{a}^{b}\frac{f(x)%
}{x}dx
\end{equation*}%
\begin{equation*}
=\frac{\left( \ln b-\ln a\right) }{4}\left[ a\dint\limits_{0}^{1}t\left( 
\frac{b}{a}\right) ^{\frac{t}{2}}f^{\prime }\left( a^{1-t}\left( ab\right) ^{%
\frac{t}{2}}\right) dt-b\dint\limits_{0}^{1}t\left( \frac{a}{b}\right) ^{%
\frac{t}{2}}f^{\prime }\left( b^{1-t}\left( ab\right) ^{\frac{t}{2}}\right)
dt\right] ,
\end{equation*}%
\begin{equation*}
\frac{f(a)+f(b)}{2}-\frac{1}{\ln b-\ln a}\dint\limits_{a}^{b}\frac{f(x)}{x}dx
\end{equation*}%
\begin{equation*}
=\frac{\left( \ln b-\ln a\right) }{2}\left[ a\dint\limits_{0}^{1}t\left( 
\frac{b}{a}\right) ^{t}f^{\prime }\left( a^{1-t}b^{t}\right)
dt-b\dint\limits_{0}^{1}t\left( \frac{a}{b}\right) ^{t}f^{\prime }\left(
b^{1-t}a^{t}\right) dt\right]
\end{equation*}
\end{lemma}

\begin{theorem}
Let $f:I\subseteq \mathbb{R}_{+}\mathbb{\rightarrow R}_{+}$ be
differentiable on $I^{\circ }$, and $a,b\in I^{\circ }$ with $a<b$ and $%
f^{\prime }\in L\left[ a,b\right] .$ If $\left\vert f^{\prime }\right\vert
^{q}$ is geometrically convex on $\left[ a,b\right] $ for $q\geq 1$, then%
\begin{equation}
\left\vert f\left( \sqrt{ab}\right) -\frac{1}{\ln b-\ln a}%
\dint\limits_{a}^{b}\frac{f(x)}{x}dx\right\vert \leq \frac{\ln b-\ln a}{4}%
\left( \frac{1}{2}\right) ^{1-\frac{1}{q}}  \label{1-2}
\end{equation}%
\begin{equation*}
\times \left\{ a\left\vert f^{\prime }\left( a\right) \right\vert \left[
g_{1}\left( \alpha \left( \frac{q}{2}\right) \right) \right] ^{\frac{1}{q}%
}+b\left\vert f^{\prime }\left( b\right) \right\vert \left[ g_{1}\left(
\gamma \left( \frac{q}{2}\right) \right) \right] ^{\frac{1}{q}}\right\} ,
\end{equation*}%
\begin{equation}
\left\vert \frac{f(a)+f(b)}{2}-\frac{1}{\ln b-\ln a}\dint\limits_{a}^{b}%
\frac{f(x)}{x}dx\right\vert \leq \frac{\ln b-\ln a}{2}\left( \frac{1}{2}%
\right) ^{1-\frac{1}{q}}  \label{1-3}
\end{equation}%
\begin{equation*}
\times \left\{ a\left\vert f^{\prime }\left( a\right) \right\vert \left[
g_{1}\left( \alpha \left( q\right) \right) \right] ^{\frac{1}{q}%
}+b\left\vert f^{\prime }\left( b\right) \right\vert \left[ g_{1}\left(
\gamma \left( q\right) \right) \right] ^{\frac{1}{q}}\right\} ,
\end{equation*}%
where%
\begin{equation*}
\bigskip g_{1}(\alpha )=\left\{ 
\begin{array}{c}
\frac{1}{2},\ \ \ \ \ \ \ \ \ \ \ \alpha =1 \\ 
\frac{\alpha \ln \alpha -\alpha +1}{\left( \ln \alpha \right) ^{2}},\ \alpha
\neq 1%
\end{array}%
\right.
\end{equation*}%
and 
\begin{equation*}
\alpha \left( u\right) =\left( \frac{b\left\vert f^{\prime }(b)\right\vert }{%
a\left\vert f^{\prime }(a)\right\vert }\right) ^{u},\ \gamma \left( u\right)
=\left( \frac{a\left\vert f^{\prime }(a)\right\vert }{b\left\vert f^{\prime
}(b)\right\vert }\right) ^{u},\ u>0.
\end{equation*}
\end{theorem}

\begin{theorem}
Let $f:I\subseteq \mathbb{R}_{+}\mathbb{\rightarrow R}_{+}$ be
differentiable on $I^{\circ }$, and $a,b\in I^{\circ }$ with $a<b$ and $%
f^{\prime }\in L\left[ a,b\right] .$ If $\left\vert f^{\prime }\right\vert
^{q}$ is geometrically convex on $\left[ a,b\right] $ for $q>1$, then%
\begin{eqnarray}
&&\left\vert f\left( \sqrt{ab}\right) -\frac{1}{\ln b-\ln a}%
\dint\limits_{a}^{b}\frac{f(x)}{x}dx\right\vert  \label{1-4} \\
&\leq &\frac{\ln b-\ln a}{4}\left( \frac{q-1}{2q-1}\right) ^{1-\frac{1}{q}} 
\notag \\
&&\times \left\{ a\left\vert f^{\prime }\left( a\right) \right\vert \left[
g_{2}\left( \alpha \left( \frac{q}{2}\right) \right) \right] ^{\frac{1}{q}%
}+b\left\vert f^{\prime }\left( b\right) \right\vert \left[ g_{2}\left(
\gamma \left( \frac{q}{2}\right) \right) \right] ^{\frac{1}{q}}\right\} 
\notag
\end{eqnarray}%
\begin{eqnarray}
&&\left\vert \frac{f(a)+f(b)}{2}-\frac{1}{\ln b-\ln a}\dint\limits_{a}^{b}%
\frac{f(x)}{x}dx\right\vert  \label{1-5} \\
&\leq &\frac{\ln b-\ln a}{2}\left( \frac{q-1}{2q-1}\right) ^{1-\frac{1}{q}} 
\notag \\
&&\times \left\{ a\left\vert f^{\prime }\left( a\right) \right\vert \left[
g_{2}\left( \alpha \left( q\right) \right) \right] ^{\frac{1}{q}%
}+b\left\vert f^{\prime }\left( b\right) \right\vert \left[ g_{2}\left(
\gamma \left( q\right) \right) \right] ^{\frac{1}{q}}\right\}  \notag
\end{eqnarray}%
where 
\begin{equation*}
g_{2}(\alpha )=\left\{ 
\begin{array}{c}
1,\ \ \ \ \alpha =1 \\ 
\frac{\alpha -1}{\ln \alpha },\ \alpha \neq 1%
\end{array}%
\right. ,
\end{equation*}%
and $\alpha \left( u\right) $,$\ \gamma \left( u\right) $ are the same as in
(\ref{2-6}).
\end{theorem}

In this paper, we will establish some new integral inequalities of
Hermite-Hadamard-like type related to the $s$-geometrically convex functions
and then apply these inequalities to special means.

\section{Main Results}

\begin{theorem}
\label{2.1}Let $f:I\subseteq \mathbb{R}_{+}\mathbb{\rightarrow R}_{+}$ be
differentiable on $I^{\circ }$, and $a,b\in I^{\circ }$ with $a<b$ and $%
f^{\prime }\in L\left[ a,b\right] .$ If $\left\vert f^{\prime }\right\vert
^{q}$ is $s$-geometrically convex on $\left[ a,b\right] $ for $q\geq 1$ and $%
s\in \left( 0,1\right] ,$ then%
\begin{equation}
\left\vert \frac{f(a)+f(b)}{2}-\frac{1}{\ln b-\ln a}\dint\limits_{a}^{b}%
\frac{f(x)}{x}dx\right\vert \leq \ln \left( \frac{b}{a}\right) \left( \frac{1%
}{2}\right) ^{2-\frac{1}{q}}H_{1}\left( s,q;g_{1}(\theta _{1}),g_{1}(\theta
_{2})\right) ,  \label{2-1}
\end{equation}%
\begin{equation}
\left\vert f\left( \sqrt{ab}\right) -\frac{1}{\ln b-\ln a}%
\dint\limits_{a}^{b}\frac{f(x)}{x}dx\right\vert \leq \ln \left( \frac{b}{a}%
\right) \left( \frac{1}{2}\right) ^{3-\frac{1}{q}}H_{1}\left(
s,q;g_{1}(\theta _{3}),g_{1}(\theta _{4})\right) ,  \label{2-2}
\end{equation}%
where 
\begin{equation*}
g_{1}(u)=\left\{ 
\begin{array}{c}
\frac{1}{2},\ \ \ \ \ \ \ \ \ \ \ u=1 \\ 
\frac{u\ln u-u+1}{\left( \ln u\right) ^{2}},\ u\neq 1%
\end{array}%
\right. ,
\end{equation*}%
\begin{eqnarray}
\theta _{1} &=&\left( \frac{b\left\vert f^{\prime }(b)\right\vert ^{s}}{%
a\left\vert f^{\prime }(a)\right\vert ^{s}}\right) ^{q},\ \theta _{2}=\left( 
\frac{a\left\vert f^{\prime }(a)\right\vert ^{s}}{b\left\vert f^{\prime
}(b)\right\vert ^{s}}\right) ^{q},  \label{2-3} \\
\theta _{3} &=&\sqrt{\theta _{1}},\ \theta _{4}=\sqrt{\theta _{2}},\   \notag
\end{eqnarray}%
\begin{eqnarray}
&&H_{1}\left( s,q;g_{1}(\theta _{i}),g_{1}(\theta _{j})\right)  \label{2-4}
\\
&=&\left\{ 
\begin{array}{c}
a\left\vert f^{\prime }(a)\right\vert ^{s}g_{1}^{1/q}\left( \theta
_{i}\right) +b\left\vert f^{\prime }(b)\right\vert ^{s}g_{1}^{1/q}\left(
\theta _{j}\right) , \\ 
\ \left\vert f^{\prime }(a)\right\vert ,\ \left\vert f^{\prime
}(b)\right\vert \leq 1, \\ 
a\left\vert f^{\prime }\left( a\right) \right\vert \left\vert f^{\prime
}\left( b\right) \right\vert ^{1-s}g_{1}^{1/q}\left( \theta _{i}\right)
+b\left\vert f^{\prime }\left( b\right) \right\vert \left\vert f^{\prime
}\left( a\right) \right\vert ^{1-s}g_{1}^{1/q}\left( \theta _{j}\right) , \\ 
\ \left\vert f^{\prime }(a)\right\vert ,\ \left\vert f^{\prime
}(b)\right\vert \geq 1, \\ 
a\left\vert f^{\prime }\left( a\right) \right\vert ^{s}\left\vert f^{\prime
}\left( b\right) \right\vert ^{1-s}g_{1}^{1/q}\left( \theta _{i}\right)
+b\left\vert f^{\prime }(b)\right\vert g_{1}^{1/q}\left( \theta _{j}\right) ,
\\ 
\ \left\vert f^{\prime }(a)\right\vert \leq 1\leq \ \left\vert f^{\prime
}(b)\right\vert , \\ 
a\left\vert f^{\prime }(a)\right\vert g_{1}^{1/q}\left( \theta _{i}\right)
+b\left\vert f^{\prime }\left( b\right) \right\vert ^{s}\left\vert f^{\prime
}\left( a\right) \right\vert ^{1-s}g_{1}^{1/q}\left( \theta _{j}\right) , \\ 
\ \left\vert f^{\prime }(b)\right\vert \leq 1\leq \ \left\vert f^{\prime
}(a)\right\vert .%
\end{array}%
\right. ,i=1,3,\ j=2,4  \notag
\end{eqnarray}
\end{theorem}

\bigskip

\begin{proof}
(1) Since $\left\vert f^{\prime }\right\vert ^{q}$ is $s$-geometrically
convex on $\left[ a,b\right] $, from lemma \ref{1.1} and power mean
inequality, we have%
\begin{eqnarray*}
&&\left\vert \frac{f(a)+f(b)}{2}-\frac{1}{\ln b-\ln a}\dint\limits_{a}^{b}%
\frac{f(x)}{x}dx\right\vert \\
&\leq &\frac{\ln \left( \frac{b}{a}\right) }{2}\left[ a\dint\limits_{0}^{1}t%
\left( \frac{b}{a}\right) ^{t}\left\vert f^{\prime }\left(
a^{1-t}b^{t}\right) \right\vert dt+b\dint\limits_{0}^{1}t\left( \frac{a}{b}%
\right) ^{t}\left\vert f^{\prime }\left( b^{1-t}a^{t}\right) \right\vert dt%
\right]
\end{eqnarray*}%
\begin{eqnarray*}
&\leq &\frac{a\ln \left( \frac{b}{a}\right) }{2}\left(
\dint\limits_{0}^{1}tdt\right) ^{1-\frac{1}{q}}\left(
\dint\limits_{0}^{1}t\left( \frac{b}{a}\right) ^{qt}\left\vert f^{\prime
}\left( a^{1-t}b^{t}\right) \right\vert ^{q}dt\right) ^{\frac{1}{q}} \\
&&+\frac{b\ln \left( \frac{b}{a}\right) }{2}\left(
\dint\limits_{0}^{1}tdt\right) ^{1-\frac{1}{q}}\left(
\dint\limits_{0}^{1}t\left( \frac{a}{b}\right) ^{qt}\left\vert f^{\prime
}\left( b^{1-t}a^{t}\right) \right\vert ^{q}dt\right) ^{\frac{1}{q}}
\end{eqnarray*}%
\begin{eqnarray}
&\leq &\frac{a\ln \left( \frac{b}{a}\right) }{2}\left( \frac{1}{2}\right)
^{1-\frac{1}{q}}\left( \dint\limits_{0}^{1}t\left( \frac{b}{a}\right)
^{qt}\left\vert f^{\prime }\left( a\right) \right\vert ^{q\left( 1-t\right)
^{s}}\left\vert f^{\prime }\left( b\right) \right\vert ^{qt^{s}}dt\right) ^{%
\frac{1}{q}}  \notag \\
&&+\frac{b\ln \left( \frac{b}{a}\right) }{2}\left( \frac{1}{2}\right) ^{1-%
\frac{1}{q}}\left( \dint\limits_{0}^{1}t\left( \frac{a}{b}\right)
^{qt}\left\vert f^{\prime }\left( b\right) \right\vert ^{q\left( 1-t\right)
^{s}}\left\vert f^{\prime }\left( a\right) \right\vert ^{qt^{s}}dt\right) ^{%
\frac{1}{q}}.  \label{2-5}
\end{eqnarray}%
If $0<\mu \leq 1\leq \eta ,\ 0<\alpha ,s\leq 1,$ then%
\begin{equation}
\mu ^{\alpha ^{s}}\leq \mu ^{\alpha s},\ \ \eta ^{\alpha ^{s}}\leq \eta
^{\alpha s+1-s}.  \label{2-a}
\end{equation}

(i) \ If $1\geq \left\vert f^{\prime }(a)\right\vert ,\ \left\vert f^{\prime
}(b)\right\vert ,$ by (\ref{2-a}) we obtain that%
\begin{eqnarray*}
&&\dint\limits_{0}^{1}t\left( \frac{b}{a}\right) ^{qt}\left\vert f^{\prime
}\left( a\right) \right\vert ^{q\left( 1-t\right) ^{s}}\left\vert f^{\prime
}\left( b\right) \right\vert ^{qt^{s}}dt \\
&\leq &\dint\limits_{0}^{1}t\left( \frac{b}{a}\right) ^{qt}\left\vert
f^{\prime }\left( a\right) \right\vert ^{qs\left( 1-t\right) }\left\vert
f^{\prime }\left( b\right) \right\vert ^{qst}dt=\left\vert f^{\prime }\left(
a\right) \right\vert ^{qs}g_{1}\left( \theta _{1}\right) ,
\end{eqnarray*}%
\begin{eqnarray}
&&\dint\limits_{0}^{1}t\left( \frac{a}{b}\right) ^{qt}\left\vert f^{\prime
}\left( b\right) \right\vert ^{q\left( 1-t\right) ^{s}}\left\vert f^{\prime
}\left( a\right) \right\vert ^{qt^{s}}dt  \label{2-6} \\
&\leq &\dint\limits_{0}^{1}t\left( \frac{a}{b}\right) ^{qt}\left\vert
f^{\prime }\left( b\right) \right\vert ^{qs\left( 1-t\right) }\left\vert
f^{\prime }\left( a\right) \right\vert ^{qst}dt=\left\vert f^{\prime }\left(
b\right) \right\vert ^{qs}g_{1}\left( \theta _{2}\right) .  \notag
\end{eqnarray}

(ii) \ If $1\leq \left\vert f^{\prime }(a)\right\vert ,\ \left\vert
f^{\prime }(b)\right\vert ,$ by (\ref{2-a}) we obtain that%
\begin{eqnarray*}
&&\dint\limits_{0}^{1}t\left( \frac{b}{a}\right) ^{qt}\left\vert f^{\prime
}\left( a\right) \right\vert ^{q\left( 1-t\right) ^{s}}\left\vert f^{\prime
}\left( b\right) \right\vert ^{qt^{s}}dt \\
&\leq &\dint\limits_{0}^{1}t\left( \frac{b}{a}\right) ^{qt}\left\vert
f^{\prime }\left( a\right) \right\vert ^{q\left( 1-st\right) }\left\vert
f^{\prime }\left( b\right) \right\vert ^{q\left( st+1-s\right) }dt \\
&=&\left( \left\vert f^{\prime }\left( a\right) \right\vert \left\vert
f^{\prime }\left( b\right) \right\vert ^{1-s}\right) ^{q}g_{1}\left( \theta
_{1}\right) ,
\end{eqnarray*}%
\begin{eqnarray}
&&\dint\limits_{0}^{1}t\left( \frac{a}{b}\right) ^{qt}\left\vert f^{\prime
}\left( b\right) \right\vert ^{q\left( 1-t\right) ^{s}}\left\vert f^{\prime
}\left( a\right) \right\vert ^{qt^{s}}dt  \label{2-7} \\
&\leq &\dint\limits_{0}^{1}t\left( \frac{a}{b}\right) ^{qt}\left\vert
f^{\prime }\left( b\right) \right\vert ^{q\left( 1-st\right) }\left\vert
f^{\prime }\left( a\right) \right\vert ^{q\left( st+1-s\right) }dt  \notag \\
&=&\left( \left\vert f^{\prime }\left( b\right) \right\vert \left\vert
f^{\prime }\left( a\right) \right\vert ^{1-s}\right) ^{q}g_{1}\left( \theta
_{2}\right) .  \notag
\end{eqnarray}%
(iii) \ If $\left\vert f^{\prime }(a)\right\vert \leq 1\leq \ \left\vert
f^{\prime }(b)\right\vert ,$ by (\ref{2-a}) we obtain that%
\begin{eqnarray*}
&&\dint\limits_{0}^{1}t\left( \frac{b}{a}\right) ^{qt}\left\vert f^{\prime
}\left( a\right) \right\vert ^{q\left( 1-t\right) ^{s}}\left\vert f^{\prime
}\left( b\right) \right\vert ^{qt^{s}}dt \\
&\leq &\dint\limits_{0}^{1}t\left( \frac{b}{a}\right) ^{qt}\left\vert
f^{\prime }\left( a\right) \right\vert ^{qs\left( 1-t\right) }\left\vert
f^{\prime }\left( b\right) \right\vert ^{q(st+1-s)}dt \\
&=&\left( \left\vert f^{\prime }\left( a\right) \right\vert ^{s}\left\vert
f^{\prime }\left( b\right) \right\vert ^{1-s}\right) ^{q}g_{1}\left( \theta
_{1}\right) ,
\end{eqnarray*}%
\begin{eqnarray}
&&\dint\limits_{0}^{1}t\left( \frac{a}{b}\right) ^{qt}\left\vert f^{\prime
}\left( b\right) \right\vert ^{q\left( 1-t\right) ^{s}}\left\vert f^{\prime
}\left( a\right) \right\vert ^{qt^{s}}dt  \label{2-8} \\
&\leq &\dint\limits_{0}^{1}t\left( \frac{a}{b}\right) ^{qt}\left\vert
f^{\prime }\left( b\right) \right\vert ^{q\left( 1-st\right) }\left\vert
f^{\prime }\left( a\right) \right\vert ^{qst}dt=\left\vert f^{\prime }\left(
b\right) \right\vert ^{q}g_{1}\left( \theta _{2}\right) .  \notag
\end{eqnarray}%
(iv) \ If $\left\vert f^{\prime }(b)\right\vert \leq 1\leq \ \left\vert
f^{\prime }(a)\right\vert ,$ by (\ref{2-a}) we obtain that%
\begin{eqnarray*}
&&\dint\limits_{0}^{1}t\left( \frac{b}{a}\right) ^{qt}\left\vert f^{\prime
}\left( a\right) \right\vert ^{q\left( 1-t\right) ^{s}}\left\vert f^{\prime
}\left( b\right) \right\vert ^{qt^{s}}dt \\
&\leq &\dint\limits_{0}^{1}t\left( \frac{b}{a}\right) ^{qt}\left\vert
f^{\prime }\left( a\right) \right\vert ^{q\left( 1-st\right) }\left\vert
f^{\prime }\left( b\right) \right\vert ^{qst}dt=\left\vert f^{\prime }\left(
a\right) \right\vert ^{q}g_{1}\left( \theta _{1}\right) ,
\end{eqnarray*}%
\begin{eqnarray}
&&\dint\limits_{0}^{1}t\left( \frac{a}{b}\right) ^{qt}\left\vert f^{\prime
}\left( b\right) \right\vert ^{q\left( 1-t\right) ^{s}}\left\vert f^{\prime
}\left( a\right) \right\vert ^{qt^{s}}dt  \label{2-9} \\
&\leq &\dint\limits_{0}^{1}t\left( \frac{a}{b}\right) ^{qt}\left\vert
f^{\prime }\left( b\right) \right\vert ^{qs\left( 1-t\right) }\left\vert
f^{\prime }\left( a\right) \right\vert ^{q(st+1-s)}dt=\left( \left\vert
f^{\prime }\left( b\right) \right\vert ^{s}\left\vert f^{\prime }\left(
a\right) \right\vert ^{1-s}\right) ^{q}g_{1}\left( \theta _{2}\right) . 
\notag
\end{eqnarray}%
From (\ref{2-5}) to (\ref{2-9}), (\ref{2-1}) holds.

(2) Since $\left\vert f^{\prime }\right\vert ^{q}$ is $s$-geometrically
convex on $\left[ a,b\right] $, from lemma \ref{2.1} and power mean
inequality, we have%
\begin{eqnarray*}
&&\left\vert f\left( \sqrt{ab}\right) -\frac{1}{\ln b-\ln a}%
\dint\limits_{a}^{b}\frac{f(x)}{x}dx\right\vert \\
&\leq &\frac{\ln \frac{b}{a}}{4}\left[ a\dint\limits_{0}^{1}t\left( \frac{b}{%
a}\right) ^{\frac{t}{2}}\left\vert f^{\prime }\left( a^{1-t}\left( ab\right)
^{\frac{t}{2}}\right) \right\vert dt+b\dint\limits_{0}^{1}t\left( \frac{a}{b}%
\right) ^{\frac{t}{2}}\left\vert f^{\prime }\left( b^{1-t}\left( ab\right) ^{%
\frac{t}{2}}\right) \right\vert dt\right]
\end{eqnarray*}%
\begin{eqnarray*}
&\leq &\frac{a\ln \frac{b}{a}}{4}\left( \dint\limits_{0}^{1}tdt\right) ^{1-%
\frac{1}{q}}\left( \dint\limits_{0}^{1}t\left( \frac{b}{a}\right) ^{\frac{qt%
}{2}}\left\vert f^{\prime }\left( a^{1-t}\left( ab\right) ^{\frac{t}{2}%
}\right) \right\vert ^{q}dt\right) ^{\frac{1}{q}} \\
&&+\frac{b\ln \frac{b}{a}}{4}\left( \dint\limits_{0}^{1}tdt\right) ^{1-\frac{%
1}{q}}\left( \dint\limits_{0}^{1}t\left( \frac{a}{b}\right) ^{\frac{qt}{2}%
}\left\vert f^{\prime }\left( b^{1-t}\left( ab\right) ^{\frac{t}{2}}\right)
\right\vert ^{q}dt\right) ^{\frac{1}{q}}
\end{eqnarray*}%
\begin{eqnarray}
&\leq &\frac{a\ln \frac{b}{a}}{4}\left( \frac{1}{2}\right) ^{1-\frac{1}{q}%
}\left( \dint\limits_{0}^{1}t\left( \frac{b}{a}\right) ^{\frac{qt}{2}%
}\left\vert f^{\prime }\left( b\right) \right\vert ^{q\left( t/2\right)
^{s}}\left\vert f^{\prime }\left( a\right) \right\vert ^{q\left( \left(
2-t\right) /2\right) ^{s}}dt\right) ^{\frac{1}{q}}  \notag \\
&&+\frac{b\ln \frac{b}{a}}{4}\left( \frac{1}{2}\right) ^{1-\frac{1}{q}%
}\left( \dint\limits_{0}^{1}t\left( \frac{a}{b}\right) ^{\frac{qt}{2}%
}\left\vert f^{\prime }\left( a\right) \right\vert ^{q\left( t/2\right)
^{s}}\left\vert f^{\prime }\left( b\right) \right\vert ^{q\left( \left(
2-t\right) /2\right) ^{s}}dt\right) ^{\frac{1}{q}}.  \label{2-10}
\end{eqnarray}

(i) \ If $1\geq \left\vert f^{\prime }(a)\right\vert ,\ \left\vert f^{\prime
}(b)\right\vert ,$ by (\ref{2-a}) we obtain that%
\begin{equation*}
\dint\limits_{0}^{1}t\left( \frac{b}{a}\right) ^{\frac{qt}{2}}\left\vert
f^{\prime }\left( b\right) \right\vert ^{q\left( t/2\right) ^{s}}\left\vert
f^{\prime }\left( a\right) \right\vert ^{q\left( \left( 2-t\right) /2\right)
^{s}}dt\leq \left\vert f^{\prime }\left( a\right) \right\vert
^{qs}g_{1}\left( \theta _{3}\right) ,
\end{equation*}%
\begin{equation}
\dint\limits_{0}^{1}t\left( \frac{a}{b}\right) ^{\frac{qt}{2}}\left\vert
f^{\prime }\left( a\right) \right\vert ^{q\left( t/2\right) ^{s}}\left\vert
f^{\prime }\left( b\right) \right\vert ^{q\left( \left( 2-t\right) /2\right)
^{s}}dt\leq \left\vert f^{\prime }\left( b\right) \right\vert
^{qs}g_{1}\left( \theta _{4}\right) .  \label{2-11}
\end{equation}

(ii) \ If $1\leq \left\vert f^{\prime }(a)\right\vert ,\ \left\vert
f^{\prime }(b)\right\vert ,$ by (\ref{2-a}) we obtain that%
\begin{equation*}
\dint\limits_{0}^{1}t\left( \frac{b}{a}\right) ^{\frac{qt}{2}}\left\vert
f^{\prime }\left( b\right) \right\vert ^{q\left( t/2\right) ^{s}}\left\vert
f^{\prime }\left( a\right) \right\vert ^{q\left( \left( 2-t\right) /2\right)
^{s}}dt\leq \left( \left\vert f^{\prime }\left( a\right) \right\vert
\left\vert f^{\prime }\left( b\right) \right\vert ^{1-s}\right)
^{q}g_{1}\left( \theta _{3}\right) ,
\end{equation*}%
\begin{equation}
\dint\limits_{0}^{1}t\left( \frac{a}{b}\right) ^{\frac{qt}{2}}\left\vert
f^{\prime }\left( a\right) \right\vert ^{q\left( t/2\right) ^{s}}\left\vert
f^{\prime }\left( b\right) \right\vert ^{q\left( \left( 2-t\right) /2\right)
^{s}}dt\leq \left( \left\vert f^{\prime }\left( b\right) \right\vert
\left\vert f^{\prime }\left( a\right) \right\vert ^{1-s}\right)
^{q}g_{1}\left( \theta _{4}\right) .  \label{2-12}
\end{equation}%
(iii) \ If $\left\vert f^{\prime }(a)\right\vert \leq 1\leq \ \left\vert
f^{\prime }(b)\right\vert ,$ by (\ref{2-a}) we obtain that%
\begin{equation*}
\dint\limits_{0}^{1}t\left( \frac{b}{a}\right) ^{\frac{qt}{2}}\left\vert
f^{\prime }\left( b\right) \right\vert ^{q\left( t/2\right) ^{s}}\left\vert
f^{\prime }\left( a\right) \right\vert ^{q\left( \left( 2-t\right) /2\right)
^{s}}dt\leq \left( \left\vert f^{\prime }\left( a\right) \right\vert
\left\vert f^{\prime }\left( b\right) \right\vert ^{1-s}\right)
^{q}g_{1}\left( \theta _{3}\right) ,
\end{equation*}%
\begin{equation}
\dint\limits_{0}^{1}t\left( \frac{a}{b}\right) ^{\frac{qt}{2}}\left\vert
f^{\prime }\left( a\right) \right\vert ^{q\left( t/2\right) ^{s}}\left\vert
f^{\prime }\left( b\right) \right\vert ^{q\left( \left( 2-t\right) /2\right)
^{s}}dt\leq \left\vert f^{\prime }\left( b\right) \right\vert
^{q}g_{1}\left( \theta _{4}\right) .  \label{2-13}
\end{equation}%
(iv) \ If $\left\vert f^{\prime }(b)\right\vert \leq 1\leq \ \left\vert
f^{\prime }(a)\right\vert ,$ by (\ref{2-a}) we obtain that%
\begin{equation*}
\dint\limits_{0}^{1}t\left( \frac{b}{a}\right) ^{\frac{qt}{2}}\left\vert
f^{\prime }\left( b\right) \right\vert ^{q\left( t/2\right) ^{s}}\left\vert
f^{\prime }\left( a\right) \right\vert ^{q\left( \left( 2-t\right) /2\right)
^{s}}dt\leq \left\vert f^{\prime }\left( a\right) \right\vert
^{q}g_{1}\left( \theta _{3}\right) ,
\end{equation*}%
\begin{equation}
\dint\limits_{0}^{1}t\left( \frac{a}{b}\right) ^{\frac{qt}{2}}\left\vert
f^{\prime }\left( a\right) \right\vert ^{q\left( t/2\right) ^{s}}\left\vert
f^{\prime }\left( b\right) \right\vert ^{q\left( \left( 2-t\right) /2\right)
^{s}}dt\leq \left( \left\vert f^{\prime }\left( b\right) \right\vert
^{s}\left\vert f^{\prime }\left( a\right) \right\vert ^{1-s}\right)
^{q}g_{1}\left( \theta _{4}\right) .  \label{2-14}
\end{equation}%
From (\ref{2-10}) to (\ref{2-14}), (\ref{2-2}) holds. This completes the
required proof.
\end{proof}

If taking $s=1$ in Theorem \ref{2.1}, we can derive the following
inequalities which are the same of the inequalities (\ref{1-2}) and (\ref%
{1-3}).

\begin{corollary}
Let $f:I\subseteq \mathbb{R}_{+}\mathbb{\rightarrow R}_{+}$ be
differentiable on $I^{\circ }$, and $a,b\in I^{\circ }$ with $a<b$ and $%
f^{\prime }\in L\left[ a,b\right] .$ If $\left\vert f^{\prime }\right\vert
^{q}$ is geometrically convex on $\left[ a,b\right] $ for $q\geq 1$, then%
\begin{equation*}
\left\vert \frac{f(a)+f(b)}{2}-\frac{1}{\ln b-\ln a}\dint\limits_{a}^{b}%
\frac{f(x)}{x}dx\right\vert \leq \ln \left( \frac{b}{a}\right) \left( \frac{1%
}{2}\right) ^{2-\frac{1}{q}}H_{1}\left( 1,q;g_{1}(\theta _{1}),g_{1}(\theta
_{2})\right) ,
\end{equation*}%
\begin{equation*}
\left\vert f\left( \sqrt{ab}\right) -\frac{1}{\ln b-\ln a}%
\dint\limits_{a}^{b}\frac{f(x)}{x}dx\right\vert \leq \ln \left( \frac{b}{a}%
\right) \left( \frac{1}{2}\right) ^{3-\frac{1}{q}}H_{1}\left(
1,q;g_{1}(\theta _{3}),g_{1}(\theta _{4})\right) ,
\end{equation*}
\end{corollary}

If taking $q=1$ in Theorem \ref{2.1}, we can derive the following corollary.

\begin{corollary}
Let $f:I\subseteq \mathbb{R}_{+}\mathbb{\rightarrow R}_{+}$ be
differentiable on $I^{\circ }$, and $a,b\in I^{\circ }$ with $a<b$ and $%
f^{\prime }\in L\left[ a,b\right] .$ If $\left\vert f^{\prime }\right\vert $
is $s$-geometrically convex on $\left[ a,b\right] $ for $s\in \left( 0,1%
\right] ,$ then%
\begin{equation*}
\left\vert \frac{f(a)+f(b)}{2}-\frac{1}{\ln b-\ln a}\dint\limits_{a}^{b}%
\frac{f(x)}{x}dx\right\vert \leq \ln \left( \frac{b}{a}\right) \left( \frac{1%
}{2}\right) ^{2-\frac{1}{q}}H_{1}\left( s,1;g_{1}(\theta _{1}),g_{1}(\theta
_{2})\right) ,
\end{equation*}%
\begin{equation*}
\left\vert f\left( \sqrt{ab}\right) -\frac{1}{\ln b-\ln a}%
\dint\limits_{a}^{b}\frac{f(x)}{x}dx\right\vert \leq \ln \left( \frac{b}{a}%
\right) \left( \frac{1}{2}\right) ^{3-\frac{1}{q}}H_{1}\left(
s,1;g_{1}(\theta _{3}),g_{1}(\theta _{4})\right) ,
\end{equation*}
\end{corollary}

\begin{theorem}
\label{2.2}Let $f:I\subseteq \mathbb{R}_{+}\mathbb{\rightarrow R}_{+}$ be
differentiable on $I^{\circ }$, and $a,b\in I^{\circ }$ with $a<b$ and $%
f^{\prime }\in L\left[ a,b\right] .$ If $\left\vert f^{\prime }\right\vert
^{q}$ is $s$-geometrically convex on $\left[ a,b\right] $ for $q>1$ and $%
s\in \left( 0,1\right] ,$ then%
\begin{equation}
\left\vert \frac{f(a)+f(b)}{2}-\frac{1}{\ln b-\ln a}\dint\limits_{a}^{b}%
\frac{f(x)}{x}dx\right\vert \leq \frac{1}{2}\ln \left( \frac{b}{a}\right)
\left( \frac{q-1}{2q-1}\right) ^{1-\frac{1}{q}}H_{2}\left( s,q;g_{2}(\theta
_{1}),g_{2}(\theta _{2})\right) ,  \label{2-15}
\end{equation}%
\begin{equation}
\left\vert f\left( \sqrt{ab}\right) -\frac{1}{\ln b-\ln a}%
\dint\limits_{a}^{b}\frac{f(x)}{x}dx\right\vert \leq \frac{1}{4}\ln \left( 
\frac{b}{a}\right) \left( \frac{q-1}{2q-1}\right) ^{1-\frac{1}{q}%
}H_{2}\left( s,q;g_{2}(\theta _{3}),g_{2}(\theta _{4})\right) ,  \label{2-16}
\end{equation}%
where 
\begin{equation*}
g_{2}(u)=\left\{ 
\begin{array}{c}
1,\ \ \ \ u=1 \\ 
\frac{u-1}{\ln u},\ u\neq 1%
\end{array}%
\right. ,u>0
\end{equation*}%
\begin{eqnarray*}
&&H_{2}\left( s,q;g_{2}(\theta _{i}),g_{2}(\theta _{j})\right) \\
&=&\left\{ 
\begin{array}{c}
a\left\vert f^{\prime }(a)\right\vert ^{s}g_{2}^{1/q}\left( \theta
_{i}\right) +b\left\vert f^{\prime }(b)\right\vert ^{s}g_{2}^{1/q}\left(
\theta _{j}\right) , \\ 
\ \left\vert f^{\prime }(a)\right\vert ,\ \left\vert f^{\prime
}(b)\right\vert \leq 1, \\ 
a\left\vert f^{\prime }\left( a\right) \right\vert \left\vert f^{\prime
}\left( b\right) \right\vert ^{1-s}g_{2}^{1/q}\left( \theta _{i}\right)
+b\left\vert f^{\prime }\left( b\right) \right\vert \left\vert f^{\prime
}\left( a\right) \right\vert ^{1-s}g_{2}^{1/q}\left( \theta _{j}\right) , \\ 
\ \left\vert f^{\prime }(a)\right\vert ,\ \left\vert f^{\prime
}(b)\right\vert \geq 1, \\ 
a\left\vert f^{\prime }\left( a\right) \right\vert ^{s}\left\vert f^{\prime
}\left( b\right) \right\vert ^{1-s}g_{2}^{1/q}\left( \theta _{i}\right)
+b\left\vert f^{\prime }(b)\right\vert g_{2}^{1/q}\left( \theta _{j}\right) ,
\\ 
\ \left\vert f^{\prime }(a)\right\vert \leq 1\leq \ \left\vert f^{\prime
}(b)\right\vert , \\ 
a\left\vert f^{\prime }(a)\right\vert g_{2}^{1/q}\left( \theta _{i}\right)
+b\left\vert f^{\prime }\left( b\right) \right\vert ^{s}\left\vert f^{\prime
}\left( a\right) \right\vert ^{1-s}g_{2}^{1/q}\left( \theta _{j}\right) , \\ 
\ \left\vert f^{\prime }(b)\right\vert \leq 1\leq \ \left\vert f^{\prime
}(a)\right\vert .%
\end{array}%
\right. , \\
i &=&1,3,\ j=2,4
\end{eqnarray*}%
and $\theta _{1}$,$\ \theta _{2}$, $\theta _{3}$, $\theta _{4}$ are the same
as in (\ref{2-3}).
\end{theorem}

\begin{proof}
(1) Since $\left\vert f^{\prime }\right\vert ^{q}$ is $s$-geometrically
convex on $\left[ a,b\right] $, from lemma \ref{1.1} and H\"{o}lder
inequality, we have%
\begin{eqnarray*}
&&\left\vert \frac{f(a)+f(b)}{2}-\frac{1}{\ln b-\ln a}\dint\limits_{a}^{b}%
\frac{f(x)}{x}dx\right\vert \\
&\leq &\frac{\ln \left( \frac{b}{a}\right) }{2}\left[ a\dint\limits_{0}^{1}t%
\left( \frac{b}{a}\right) ^{t}\left\vert f^{\prime }\left(
a^{1-t}b^{t}\right) \right\vert dt+b\dint\limits_{0}^{1}t\left( \frac{a}{b}%
\right) ^{t}\left\vert f^{\prime }\left( b^{1-t}a^{t}\right) \right\vert dt%
\right]
\end{eqnarray*}%
\begin{eqnarray*}
&\leq &\frac{a\ln \left( \frac{b}{a}\right) }{2}\left(
\dint\limits_{0}^{1}t^{\frac{q}{q-1}}dt\right) ^{1-\frac{1}{q}}\left(
\dint\limits_{0}^{1}\left( \frac{b}{a}\right) ^{qt}\left\vert f^{\prime
}\left( a^{1-t}b^{t}\right) \right\vert ^{q}dt\right) ^{\frac{1}{q}} \\
&&+\frac{b}{2}\ln \left( \frac{b}{a}\right) \left( \dint\limits_{0}^{1}t^{%
\frac{q}{q-1}}dt\right) ^{1-\frac{1}{q}}\left( \dint\limits_{0}^{1}\left( 
\frac{a}{b}\right) ^{qt}\left\vert f^{\prime }\left( b^{1-t}a^{t}\right)
\right\vert ^{q}dt\right) ^{\frac{1}{q}}
\end{eqnarray*}%
\begin{eqnarray}
&\leq &\frac{a\ln \left( \frac{b}{a}\right) }{2}\left( \frac{q-1}{2q-1}%
\right) ^{1-\frac{1}{q}}\left( \dint\limits_{0}^{1}\left( \frac{b}{a}\right)
^{qt}\left\vert f^{\prime }\left( b\right) \right\vert ^{qt^{s}}\left\vert
f^{\prime }\left( a\right) \right\vert ^{q\left( 1-t\right) ^{s}}dt\right) ^{%
\frac{1}{q}}  \notag \\
&&+\frac{b}{2}\ln \left( \frac{b}{a}\right) \left( \frac{q-1}{2q-1}\right)
^{1-\frac{1}{q}}\left( \dint\limits_{0}^{1}\left( \frac{a}{b}\right)
^{qt}\left\vert f^{\prime }\left( a\right) \right\vert ^{qt^{s}}\left\vert
f^{\prime }\left( b\right) \right\vert ^{q\left( 1-t\right) ^{s}}dt\right) ^{%
\frac{1}{q}}.  \label{2-17}
\end{eqnarray}

(i) \ If $1\geq \left\vert f^{\prime }(a)\right\vert ,\ \left\vert f^{\prime
}(b)\right\vert ,$ by (\ref{2-a}) we have%
\begin{equation*}
\dint\limits_{0}^{1}\left( \frac{b}{a}\right) ^{qt}\left\vert f^{\prime
}\left( b\right) \right\vert ^{qt^{s}}\left\vert f^{\prime }\left( a\right)
\right\vert ^{q\left( 1-t\right) ^{s}}dt=\left\vert f^{\prime }\left(
a\right) \right\vert ^{qs}g_{2}\left( \theta _{1}\right) ,
\end{equation*}%
\begin{equation}
\dint\limits_{0}^{1}\left( \frac{a}{b}\right) ^{qt}\left\vert f^{\prime
}\left( a\right) \right\vert ^{qt^{s}}\left\vert f^{\prime }\left( b\right)
\right\vert ^{q\left( 1-t\right) ^{s}}dt=\left\vert f^{\prime }\left(
b\right) \right\vert ^{qs}g_{2}\left( \theta _{2}\right) .  \label{2-18}
\end{equation}

(ii) \ If $1\leq \left\vert f^{\prime }(a)\right\vert ,\ \left\vert
f^{\prime }(b)\right\vert ,$ by (\ref{2-a}) we have%
\begin{equation*}
\dint\limits_{0}^{1}\left( \frac{b}{a}\right) ^{qt}\left\vert f^{\prime
}\left( b\right) \right\vert ^{qt^{s}}\left\vert f^{\prime }\left( a\right)
\right\vert ^{q\left( 1-t\right) ^{s}}dt\leq \left( \left\vert f^{\prime
}\left( a\right) \right\vert \left\vert f^{\prime }\left( b\right)
\right\vert ^{1-s}\right) ^{q}g_{2}\left( \theta _{1}\right) ,
\end{equation*}%
\begin{equation}
\dint\limits_{0}^{1}\left( \frac{a}{b}\right) ^{qt}\left\vert f^{\prime
}\left( a\right) \right\vert ^{qt^{s}}\left\vert f^{\prime }\left( b\right)
\right\vert ^{q\left( 1-t\right) ^{s}}dt\leq \left( \left\vert f^{\prime
}\left( b\right) \right\vert \left\vert f^{\prime }\left( a\right)
\right\vert ^{1-s}\right) ^{q}g_{2}\left( \theta _{2}\right) .  \label{2-19}
\end{equation}%
(iii) \ If $\left\vert f^{\prime }(a)\right\vert \leq 1\leq \ \left\vert
f^{\prime }(b)\right\vert ,$ by (\ref{2-a}) we obtain that%
\begin{equation*}
\dint\limits_{0}^{1}\left( \frac{b}{a}\right) ^{qt}\left\vert f^{\prime
}\left( b\right) \right\vert ^{qt^{s}}\left\vert f^{\prime }\left( a\right)
\right\vert ^{q\left( 1-t\right) ^{s}}dt\leq \left( \left\vert f^{\prime
}\left( a\right) \right\vert ^{s}\left\vert f^{\prime }\left( b\right)
\right\vert ^{1-s}\right) ^{q}g_{2}\left( \theta _{1}\right) ,
\end{equation*}%
\begin{equation}
\dint\limits_{0}^{1}\left( \frac{a}{b}\right) ^{qt}\left\vert f^{\prime
}\left( a\right) \right\vert ^{qt^{s}}\left\vert f^{\prime }\left( b\right)
\right\vert ^{q\left( 1-t\right) ^{s}}dt\leq \left\vert f^{\prime }\left(
b\right) \right\vert ^{q}g_{2}\left( \theta _{2}\right) .  \label{2-20}
\end{equation}%
(iv) \ If $\left\vert f^{\prime }(b)\right\vert \leq 1\leq \ \left\vert
f^{\prime }(a)\right\vert ,$ by (\ref{2-a}) we obtain that%
\begin{equation*}
\dint\limits_{0}^{1}\left( \frac{b}{a}\right) ^{qt}\left\vert f^{\prime
}\left( b\right) \right\vert ^{qt^{s}}\left\vert f^{\prime }\left( a\right)
\right\vert ^{q\left( 1-t\right) ^{s}}dt\leq \left\vert f^{\prime }\left(
a\right) \right\vert ^{q}g_{2}\left( \theta _{1}\right) ,
\end{equation*}%
\begin{equation}
\dint\limits_{0}^{1}\left( \frac{a}{b}\right) ^{qt}\left\vert f^{\prime
}\left( a\right) \right\vert ^{qt^{s}}\left\vert f^{\prime }\left( b\right)
\right\vert ^{q\left( 1-t\right) ^{s}}dt\leq \left( \left\vert f^{\prime
}\left( b\right) \right\vert ^{s}\left\vert f^{\prime }\left( a\right)
\right\vert ^{1-s}\right) ^{q}g_{2}\left( \theta _{2}\right) .  \label{2-21}
\end{equation}%
From (\ref{2-17}) to (\ref{2-21}), (\ref{2-15}) holds.

(2) Since $\left\vert f^{\prime }\right\vert ^{q}$ is $s$-geometrically
convex on $\left[ a,b\right] $, from lemma \ref{2.1} and H\"{o}lder
inequality, we have%
\begin{eqnarray*}
&&\left\vert f\left( \sqrt{ab}\right) -\frac{1}{\ln b-\ln a}%
\dint\limits_{a}^{b}\frac{f(x)}{x}dx\right\vert \\
&\leq &\frac{\ln \frac{b}{a}}{4}\left[ a\dint\limits_{0}^{1}t\left( \frac{b}{%
a}\right) ^{\frac{t}{2}}\left\vert f^{\prime }\left( a^{1-t}\left( ab\right)
^{\frac{t}{2}}\right) \right\vert dt+b\dint\limits_{0}^{1}t\left( \frac{a}{b}%
\right) ^{\frac{t}{2}}\left\vert f^{\prime }\left( b^{1-t}\left( ab\right) ^{%
\frac{t}{2}}\right) \right\vert dt\right]
\end{eqnarray*}%
\begin{eqnarray}
&\leq &\frac{a\ln \frac{b}{a}}{4}\left( \frac{q-1}{2q-1}\right) ^{1-\frac{1}{%
q}}\left( \dint\limits_{0}^{1}\left( \frac{b}{a}\right) ^{\frac{qt}{2}%
}\left\vert f^{\prime }\left( b\right) \right\vert ^{q\left( t/2\right)
^{s}}\left\vert f^{\prime }\left( a\right) \right\vert ^{q\left( \left(
2-t\right) /2\right) ^{s}}dt\right) ^{\frac{1}{q}}  \notag \\
&&+\frac{b\ln \frac{b}{a}}{4}\left( \frac{q-1}{2q-1}\right) ^{1-\frac{1}{q}%
}\left( \dint\limits_{0}^{1}\left( \frac{a}{b}\right) ^{\frac{qt}{2}%
}\left\vert f^{\prime }\left( a\right) \right\vert ^{q\left( t/2\right)
^{s}}\left\vert f^{\prime }\left( b\right) \right\vert ^{q\left( \left(
2-t\right) /2\right) ^{s}}dt\right) ^{\frac{1}{q}}.  \label{2-22}
\end{eqnarray}%
From (\ref{2-22}) and similarly to (\ref{2-18}) - (\ref{2-21}), (\ref{2-16})
holds.
\end{proof}

If taking $s=1$ in Theorem \ref{2.2}, we can derive the following
inequalities which are the same of the inequalities (\ref{1-4}) and (\ref%
{1-5}).

\begin{corollary}
Let $f:I\subseteq \mathbb{R}_{+}\mathbb{\rightarrow R}_{+}$ be
differentiable on $I^{\circ }$, and $a,b\in I^{\circ }$ with $a<b$ and $%
f^{\prime }\in L\left[ a,b\right] .$ If $\left\vert f^{\prime }\right\vert
^{q}$ is geometrically convex on $\left[ a,b\right] $ for $q>1,$ then%
\begin{equation*}
\left\vert \frac{f(a)+f(b)}{2}-\frac{1}{\ln b-\ln a}\dint\limits_{a}^{b}%
\frac{f(x)}{x}dx\right\vert \leq \frac{1}{2}\ln \left( \frac{b}{a}\right)
\left( \frac{q-1}{2q-1}\right) ^{1-\frac{1}{q}}H_{2}\left( 1,q;g_{2}(\theta
_{1}),g_{2}(\theta _{2})\right) ,
\end{equation*}%
\begin{equation*}
\left\vert f\left( \sqrt{ab}\right) -\frac{1}{\ln b-\ln a}%
\dint\limits_{a}^{b}\frac{f(x)}{x}dx\right\vert \leq \frac{1}{4}\ln \left( 
\frac{b}{a}\right) \left( \frac{q-1}{2q-1}\right) ^{1-\frac{1}{q}%
}H_{2}\left( 1,q;g_{2}(\theta _{3}),g_{2}(\theta _{4})\right) .
\end{equation*}
\end{corollary}

\section{Application to Special Means}

Let us recall the following special means of two nonnegative number $a,b$
with $b>a:$

\begin{enumerate}
\item The arithmetic mean%
\begin{equation*}
A=A\left( a,b\right) :=\frac{a+b}{2}.
\end{equation*}

\item The geometric mean%
\begin{equation*}
G=G\left( a,b\right) :=\sqrt{ab}.
\end{equation*}

\item The logarithmic mean%
\begin{equation*}
L=L\left( a,b\right) :=\frac{b-a}{\ln b-\ln a}.
\end{equation*}

\item The p-logarithmic mean%
\begin{equation*}
L_{p}=L_{p}\left( a,b\right) :=\left( \frac{b^{p+1}-a^{p+1}}{(p+1)(b-a)}%
\right) ^{\frac{1}{p}},\ \ p\in 
%TCIMACRO{\U{211d} }%
%BeginExpansion
\mathbb{R}
%EndExpansion
\backslash \left\{ -1,0\right\} .
\end{equation*}
\end{enumerate}

Let $f(x)=\left( x^{s}/s\right) ,\ x\in \left( 0,1\right] ,\ 0<s<1,\ q\geq 1$%
\ then the function $\left\vert f^{\prime }(x)\right\vert ^{q}=x^{(s-1)q}$
is $s$-geometrically convex on $\left( 0,1\right] $ for $0<s<1$ and $%
\left\vert f^{\prime }(a)\right\vert =a^{s-1}>b^{s-1}=\left\vert f^{\prime
}(b)\right\vert \geq 1$ (see \cite{ZJQ12}).

\begin{proposition}
\label{3.1}Let $0<a<b\leq 1,\ 0<s<1,$and $q\geq 1.$ Then 
\begin{eqnarray*}
&&\left\vert A\left( a^{s},b^{s}\right) -L\left( a^{s},b^{s}\right)
\right\vert \leq \frac{s}{2G^{2(s-1)^{2}}(a,b)}\left( \frac{b-a}{2L\left(
a,b\right) }\right) ^{1-\frac{1}{q}}\left( \frac{1}{\left( s^{2}-s+1\right) q%
}\right) ^{\frac{1}{q}} \\
&&\times \left[ \left\{ b^{\left( s^{2}-s+1\right) q}-L(a^{\left(
s^{2}-s+1\right) q},b^{\left( s^{2}-s+1\right) q})\right\} ^{\frac{1}{q}%
}\right. \\
&&\left. +\left\{ L(a^{\left( s^{2}-s+1\right) q},b^{\left( s^{2}-s+1\right)
q})-a^{\left( s^{2}-s+1\right) q}\right\} ^{\frac{1}{q}}\right] ,
\end{eqnarray*}%
\begin{eqnarray*}
&&\left\vert G^{s}\left( a,b\right) -L\left( a^{s},b^{s}\right) \right\vert
\leq \frac{s}{2G^{(s-1)^{2}}(a,b)}\left( \frac{b-a}{4L\left( a,b\right) }%
\right) ^{1-\frac{1}{q}}\left( \frac{1}{\left( s^{2}-s+1\right) q}\right) ^{%
\frac{1}{q}} \\
&&\times \left[ G(a^{s},b^{-(s-1)^{2}})\left\{ b^{\left( s^{2}-s+1\right)
q/2}-L(a^{\left( s^{2}-s+1\right) q/2},b^{\left( s^{2}-s+1\right)
q/2})\right\} ^{\frac{1}{q}}\right. \\
&&\left. +G(b^{s},a^{-(s-1)^{2}})\left\{ L(a^{\left( s^{2}-s+1\right)
q/2},b^{\left( s^{2}-s+1\right) q/2})-a^{\left( s^{2}-s+1\right)
q/2}\right\} ^{\frac{1}{q}}\right] ,
\end{eqnarray*}
\end{proposition}

\begin{proof}
The assertion follows from the inequalities (\ref{2-1}) and (\ref{2-2}) in
Theorem \ref{2.1} for $f(x)=\left( x^{s}/s\right) ,\ x\in \left( 0,1\right]
,\ 0<s<1,$.
\end{proof}

\begin{proposition}
\label{3.2}Let $0<a<b\leq 1,\ 0<s<1,$and $q>1.$ Then 
\begin{eqnarray*}
&&\left\vert A\left( a^{s},b^{s}\right) -L\left( a^{s},b^{s}\right)
\right\vert \\
&\leq &\frac{s\left( b-a\right) }{L(a,b)G^{2(s-1)^{2}}(a,b)}\left( \frac{q-1%
}{2q-1}\right) ^{1-\frac{1}{q}}L^{\frac{1}{q}}(a^{\left( s^{2}-s+1\right)
q},b^{\left( s^{2}-s+1\right) q}),
\end{eqnarray*}%
\begin{eqnarray*}
&&\left\vert G^{s}\left( a,b\right) -L\left( a^{s},b^{s}\right) \right\vert
\\
&\leq &\frac{s\left( b-a\right) }{2L(a,b)G^{(s-1)^{2}}(a,b)}\left( \frac{q-1%
}{2q-1}\right) ^{1-\frac{1}{q}}L^{\frac{1}{q}}(a^{\left( s^{2}-s+1\right)
q/2},b^{\left( s^{2}-s+1\right) q/2}) \\
&&\times A\left( G\left( a^{-(s-1)^{2}},b^{s}\right) ,G\left(
a^{s},b^{-(s-1)^{2}}\right) \right) ,
\end{eqnarray*}%
\bigskip
\end{proposition}

\begin{proof}
The assertion follows from the inequalities (\ref{2-15}) and (\ref{2-16}) in
Theorem \ref{2.2} for $f(x)=\left( x^{s}/s\right) ,\ x\in \left( 0,1\right]
,\ 0<s<1,$.
\end{proof}

\end{document}